\documentclass[12pt]{article}
\usepackage{amssymb}
\usepackage{amsthm}
\usepackage{amsmath}
\usepackage{dsfont}

\theoremstyle{plain}
\newtheorem{theorem}{Theorem}
\newtheorem{lemma}[theorem]{Lemma}
\newtheorem{corollary}[theorem]{Corollary}

\title{Note on eigenvectors from eigenvalues}
\date{}
\author{Xiaomei Chen\\{\footnote{\today}}
}
\begin{document}
\maketitle

\vspace{-15mm}
\begin{abstract}
Denton, Parke, Tao and Zhang gave a new method which determines eigenvectors from eigenvalues for Hermitian matrices with distinct eigenvalues. In this short note, we extend the above result to general Hermitian matrices.
\end{abstract}
\vspace{5mm}

Let $A$ be a $n\times n$ Hermitian matrix with characteristic polynomial $f_A(\lambda)=(\lambda-\lambda_1(A))^{\mu_1}(\lambda-\lambda_2(A))^{\mu_2}\cdots (\lambda-\lambda_d(A))^{\mu_d}$, where $\mu_i$ is the algebraic multiplicity of $\lambda_i(A)$. Let $v_{i1},v_{i2},\dots,v_{i\mu_i}$ be a sequence of orthonormal eigenvectors of $A$ corresponding to $\lambda_i(A)$. Given $S\subset [n]$,we use $M_{S}$ to denote the submatrix obtained from $A$ by deleting rows and columns with indices belonging to $S$.

We first give a result about the determinant of blocks of unitary matrices.
\begin{lemma}
\label{lemma}
Let $P_{11}$ and $P_{22}$ be square matrices with order $r$ and $n-r$ respectively. If the block matrix
$$P=\left[\begin{matrix}
     P_{11} & P_{12} \\
     P_{21} & P_{22}
\end{matrix}\right]$$
is unitary, then we have $|\mathrm{det}(P_{11})|^2=|\mathrm{det}(P_{22})|^2$.
\end{lemma}
\begin{proof}
Since $P$ is unitary, we have $P^{*}P=PP^{*}=I_{n}$, which leads to
$$P_{11}^{*}P_{11}+P_{21}^{*}P_{21}=I_r,$$
$$P_{21}P_{21}^{*}+P_{22}P_{22}^{*}=I_{n-r}.$$
Thus we have
$$|\mathrm{det}(P_{11})|^2=\mathrm{det}(P_{11}^{*}P_{11})=\mathrm{det}(I_r-P_{21}^{*}P_{21}),$$
$$|\mathrm{det}(P_{22})|^2=\mathrm{det}(P_{22}P_{22}^{*})=\mathrm{det}(I_{n-r}-P_{21}P_{21}^{*}),$$
which gives the result.
\end{proof}

For $S\subset [n]$ with $|S|=\mu_i$, we use $[v_{i1},\dots,v_{i\mu_i}]_{S}$ to denote the $\mu_i\times\mu_i$ submatrix of $[v_{i1},\dots,v_{i\mu_i}]$ with $S$ the set of indices of its rows. Let $\lambda_{j}(M_S)$ denote the eigenvalues of $M_S$. Then our main result is as follows.
\begin{theorem}
\begin{equation}
\label{mainequ}
|\mathrm{det}([v_{i1},\dots,v_{i\mu_i}]_{S})|^2=\frac{\prod_{j=1}^{n-\mu_i} (\lambda_i(A)-\lambda_j(M_S))}{\prod_{j=1;j\neq i}^{d}(\lambda_i(A)-\lambda_j(A))^{\mu_j}}.
\end{equation}
\end{theorem}
\begin{proof}
WLOG we take $i=1$ and $S=\{1,2,\dots,\mu_1\}$. Since both sides of eq. \ref{mainequ} remain the same after shifting $A$ by $\lambda_1I_n$, it is enough to prove the result for $\lambda_1=0$. Then eq. \ref{mainequ} becomes
\begin{equation}
\label{eequ}
|\mathrm{det}([v_{11},\dots,v_{1\mu_1}]_{S})|^2=\frac{\prod_{j=1}^{n-\mu_1} \lambda_j(M_S)}{\prod_{j=2}^{d}\lambda_j(A)^{\mu_j}}.
\end{equation}

Let
\begin{equation*}
\begin{aligned}
P=&[v_{11},\dots,v_{1\mu_1},v_{21},\dots,v_{2\mu_{2}},\dots,v_{d1},\dots,v_{d\mu_{d}}]\\
=&\left[\begin{matrix}
     [v_{11},\dots,v_{1\mu_1}]_{S} & P_{12} \\
     P_{21} & P_{22}
\end{matrix}\right]
\end{aligned}
\end{equation*}
and
\begin{equation*}
\begin{aligned}
D=&\mathrm{diag}(\underbrace{0,\dots,0}_{\mu_1},\underbrace{\lambda_2(A),\dots,\lambda_2(A)}_{\mu_2},\dots,\underbrace{\lambda_d(A),\dots,\lambda_d(A)}_{\mu_d})\\
=&\left[\begin{matrix}
     0_{\mu_1} & 0 \\
     0 & D_1
\end{matrix}\right]
\end{aligned}
\end{equation*}
 Then we have $A=PDP^{*}$, which implies that $M_{S}=P_{22}D_1P_{22}^{*}$. By Lemma \ref{lemma}, we know that $|\mathrm{det}([v_{11},\dots,v_{1\mu_1}]_{S})|^2=\mathrm{det}(P_{22}P_{22}^{*})$. Thus we have
 $$|\mathrm{det}([v_{11},\dots,v_{1\mu_1}]_{S})|^2=\frac{\mathrm{det}(M_{S})}{\mathrm{det} (D_1)}=\frac{\prod_{j=1}^{n-\mu_1} \lambda_j(M_S)}{\prod_{j=2}^{d}\lambda_j(A)^{\mu_j}}.$$
\end{proof}

Now we can obtain Lemma 2 in \cite{[Denton]} as a special case of Theorem 2.
\begin{corollary}\cite{[Denton],[Denton1]}
If $\mu_i=1$, then we have
\begin{equation*}
|v_{i1}(k)|^2=\frac{\prod_{j=1}^{n-1} (\lambda_i(A)-\lambda_j(M_k))}{\prod_{j=1;j\neq i}^{d}(\lambda_i(A)-\lambda_j(A))^{\mu_j}},
\end{equation*}
where $v_{i1}(k)$ denotes the $k^{th}$ element of $v_{i1}$.
\end{corollary}



\begin{thebibliography}{00}
\bibitem{[Denton]}Peter B Denton, Stephen J Parke, Terence Tao, Xining Zhang. Eigenvectors from eigenvalues.Preprint, arXiv:1908.03795v1. 2019.
\bibitem{[Denton1]}Peter B Denton, Stephen J Parke, Xining Zhang. Eigenvalues: the Rosetta Stone for Neutrino Oscillations in Matter. arXiv:1907.02534. 2019.
\end{thebibliography}



\noindent{\emph{Address}: School of Mathematics and Computational Science, Hunan University of Science and Technology, Xiangtan 411201, China.}\\
\noindent{\emph{E-mail address}: xmchen@hnust.edu.cn}
\end{document}